\documentclass{article}
\usepackage{amsthm}
\usepackage{amssymb}
\usepackage{amsmath}
\usepackage{bbold}
\usepackage{dsfont}
\usepackage{mathrsfs}
\usepackage{xspace}
\usepackage{enumerate}
\usepackage[pdftex]{graphicx}
\usepackage{subcaption}
\usepackage[format = hang]{caption}
\usepackage{subcaption}
\usepackage{array}
\usepackage{thmtools, thm-restate}
\usepackage[section]{placeins}
\usepackage{adjustbox}
\usepackage{placeins}

\newcommand{\nats}{\ensuremath{\mathds{N}}\xspace}
\newcommand{\natsz}{\ensuremath{\mathds{N}_0}\xspace}

\newcommand{\CM}{\texttt{Cle\-ver\-Ma\-ker}\xspace}
\newcommand{\RM}{\texttt{Ran\-dom\-Ma\-ker}\xspace}
\newcommand{\CB}{\texttt{Cle\-ver\-Br\-eaker}\xspace}
\newcommand{\RB}{\texttt{Ran\-dom\-Br\-eaker}\xspace}
\newcommand{\CP}{\texttt{CP}\xspace}
\newcommand{\RP}{\texttt{RP}\xspace}

\newcommand{\bigchoose}[2]{\Bigl({{#1}\atop#2}\Bigr)}

\newcommand{\prob}[1]{ Pr \left[#1\right]}

\newcommand{\Gsigma}[2]{ G_{#1}\left(#2\right)}

\declaretheorem[numberwithin=section, name=Theorem]{thm}
\declaretheorem[sibling=thm, name=Lemma]{lem}

\declaretheorem[sibling=thm, name=Proposition]{prop}

\declaretheorem[sibling=thm, style=definition, name=Definition]{defin}

\title{Sharp thresholds for half-random games II}
\date{}
\author{Jonas Groschwitz\\ Universit\"at Potsdam  \and Tibor Szab\'o \\ Freie Universit\"at Berlin}

\begin{document}
	
	\maketitle
	
	\begin{abstract}
		We study biased Maker-Breaker positional games between two players, one of whom is playing 
		randomly against an opponent with an optimal strategy. In this work we focus on the case of Breaker playing randomly and Maker being ``clever''. The reverse scenario is treated 
		in a separate paper. 
		We determine the sharp threshold bias of classical games played on the edge set of the complete graph $K_n$, 
		such as connectivity, perfect matching, Hamiltonicity, and minimum degree-$1$. 
		In all of these games, the threshold is equal to the trivial upper bound implied by the 
		number of edges needed for Maker to occupy a winning set. Moreover, 
		we show that \CM can not only win against asymptotically optimal bias, but can do so very fast, wasting only logarithmically many moves (while the winning set sizes are linear in $n$).
	\end{abstract}
	
	\section{Introduction}
	
	In a Maker-Breaker positional game, two players take turns occupying a free element of a vertex set $X$, called \textit{the board}. The game is defined by a finite hypergraph 
	$\mathcal{F}\subset2^X$. One player, called Maker, is called the winner if he occupies all vertices of a hyperedge in $\mathcal{F}$. Otherwise the other player, called Breaker, wins. 
	We focus on graph games, where the board $X$ is the edge set of a complete graph $K_n$, and $\mathcal{F}$ consists of all subgraphs with a certain graph property. Here we focus on the game hypergraphs ${\cal C}(n)$, ${\cal PM}(n)$, ${\cal H}(n)$, ${\cal D}_1(n)$, ${\cal D}_2(n)$ denoting the edge sets of $n$-vertex graphs that are connected, have a perfect matching, have a hamilton cycle, have minimum degree at least one and two, respectively. Mostly, $n$ will be clear from the context and we omit it from the notation.
	
	Since the game has perfect information and no chance elements, one player has a winning strategy -- which of the two players, depends on the game. A standard method, suggested 
	by Chv\'{a}tal and Erd\H{o}s~\cite{chvaterd78},
	 to compensate this imbalance inherent to the game, is to introduce a \textit{bias}, that is, to allow the ``disadvantaged'' player to occupy more than one element per turn. In an \textit{$(a : b)$ biased positional game} Maker occupies $a$ elements per turn and Breaker $b$ elements.

	\subsection{Threshold Bias and Half-Random Games}
	In all the mentioned graph games, Maker wins rather easily with a $(1: 1)$ bias. But what bias is necessary to allow Breaker to win? Given a game $\mathcal{F}$ we define the \textit{threshold bias} $b_\mathcal{F}$ to be the smallest integer $b$ such that Breaker has a winning strategy in the $(1:b)$ biased game $\mathcal{F}$.  The threshold bias of the 
	connectivity game was already studied by Chv\'{a}tal and Erd\H{o}s~\cite{chvaterd78} in 1978, who determined it to be $b_{\cal C} = \Theta\left(\frac{n}{\ln n}\right)$. They also showed that $b_{\cal H} > 1$. 
	Subsequently the lower bounds on the threshold bias of the connectivity game, 
	as well as the one of the Hamiltonicity game was improved in a series of papers by
	Beck and several other researchers.
	Gebauer and Szab\'o \cite{gebsza09} and Krivelevich \cite{kriv09}, respectively, 
	showed that both threshold are $(1+o(1)) \frac{n}{\ln n}$.
	
	An instrumental way, suggested implicitly by Chv\'{a}tal and Erd\H{o}s~\cite{chvaterd78}, 
	to gain insight into particular positional games is to study what happens when both
	players occupy a uniformly random free edge. Interestingly, as 
	an immediate consequence of classic theorems from the theory of random graphs, 
	these random connectivity 
	and Hamiltonicity games exhibit the very same threshold asymptotics as their deterministic counterparts. This phenomenon is called the \textit{probabilistic intuition} and is a driving force behind much of the research in positional games. 
	For a more detailed discussion of the relevant history of biased graph games and the 
	probabilistic intuition, see the first part of our work \cite{THRG1}.
	
	A natural problem arising from the desire for better understanding of the probabilistic
	intuition is to examine the games with exactly one player playing randomly and the other following a (clever) strategy. There are of course two possible scenarios for these 
	\textit{half-random positional game}s: either Maker or Breaker is the one who plays randomly.	 In this paper we focus on the latter; our work on the first scenario is contained 
	in \cite{THRG1}. To signify their strategies, we call our players \RB and  \CM. 
	We show that playing randomly puts \RB at a serious disadvantage against her clever opponent: the threshold bias of the game is much higher than the $\frac{n}{\ln n}$ of  
	the purely clever and purely random game. 
	
Another aspect of positional games we emphasize in this paper is the efficiency of Maker's winning strategy. The question of winning fast has recently been the subject of vigorous research. Among others, Hefetz et al. \cite{hefetz2009fast} and later 
Hefetz and Stich~\cite{HefetzStich} established optimally fast Maker strategies for unbiased non-random Maker-Breaker games, in particular Hamiltonicity and perfect matching. 
Ferber and Hefetz~\cite{FerberHefetzI, FerberHefetzII} 
used fast winning strategies to obtain positive results for certain
 strong positional games. Strong positional games are in a sense the symmetric version 
 of Maker-Breaker games, often proving to be quite inaccessible by standard methods.
For Avoider-Enforcer games, the optimal speed of strategies was studied by Hefetz et al. \cite{hefetz2009fastAvoiderEnforcer} and Bar{\'a}t and Stojakovic \cite{barat2010winning},
among others.

		In order to state our results we define the precise notion of sharp threshold bias of a 
	\CM/\RB half-random games. In what follows, when we talk about a game, we actually mean a \emph{sequence of games} parametrized by $n$, and similarly by strategy we mean a \emph{sequence of strategies}.
	\begin{defin}
		We say a function $k:\natsz\mapsto\natsz$ is a \textit{sharp threshold bias} of the $(1:b)$
		half-random positional game between \CM and \RB, if for every $\epsilon >0$ the following two conditions are satisfied:
		\begin{enumerate}[(a)]
			\item \RB wins the $(1:(1+\epsilon) k(n))$-biased game with probability tending 
			to  $1$ against any strategy of \CM, and 
			\item \CM has a strategy against which \RB loses the $(1:(1-\epsilon)k(n))$-biased 
			game with probability tending to $1$.
		\end{enumerate}
	\end{defin}
	
	\subsection{Results}
	We establish that both the perfect matching and the Hamiltonicity game  
	have a sharp threshold bias. In both cases the sharp threshold turns out to match
	the trivial upper bound derived from the observation that a large \RB
	bias makes it impossible for \CM to occupy as many edges throughout
	the game as there are in just a single winning structure.

	If the bias of \RB is at least $n$ then \CM occupies at most 
	${n\choose 2}/(n+1) <\frac{n}{2}$ edges and hence can occupy neither a 
	perfect matching nor a graph with minimum degree $1$.
	In our first theorem we show that this trivial upper bound on the threshold biases of
	the games ${\cal PM}$ and  ${\cal D}_1$ is essentially tight. We achieve this by providing 
	\CM with a strategy that occupies a perfect matching very fast, in just $O(\log n)$ more rounds than the absolute necessary $\frac{n}{2}$.
	\begin{restatable}{thm}{rstConnCMRB}\label{thm:resPmCMRB}
		For every $\epsilon >0$, 
		\CM has a strategy in the $(1: (1-\epsilon)n))$ half-random game ${\cal PM}$
		that is winning in $\frac{n}{2} + O(\log n)$ moves a.a.s.
		In particular the sharp threshold bias for both the $(1:b)$ 
		perfect matching, and the $(1:b)$ minimum degree-$1$ half-random
		game between \CM and \RB is $n$.
	\end{restatable}
	
	For our next theorem observe that if the bias of \RB is more than $\frac{n}{2}$ 
	then \CM occupies less than ${n\choose 2}/(n/2) =n-1$ edges and hence can 
	build neither a connected graph nor a Hamilton cycle nor a graph with minimum degree $2$. 
	It turns out that this trivial upper bound on the 
	threshold biases is essentially tight for all three games. 
	Again, we prove this by providing \CM with a very fast strategy to build a Hamiltonian cycle, succeeding in just $n+O(\log n)$ moves.
	\begin{restatable}{thm}{rstHamCMRB}\label{thm:resHamCMRB}
		For every $\epsilon >0$, 
		\CM has a strategy in the $(1: (1-\epsilon)\frac{n}{2}))$ half-random game ${\cal H}$
		that is winning in $n + O(\log n)$ moves a.a.s.
		In particular the sharp threshold bias for the $(1:b)$ 
		connectivity, minimum-degree-$2$, and  Hamiltonicity half-random
		games between \CM and \RB is $\frac{n}{2}$.
	\end{restatable}
	
	Note that all the games discussed in Theorems~\ref{thm:resPmCMRB} and \ref{thm:resHamCMRB} 
	have a significantly higher half-random threshold than the threshold bias 
	$\frac{n}{\ln{n}}$ in their fully deterministic and fully random version.
	
	{\bf Remark.} The results of this paper are based on the Master thesis
	of the first author~\cite{Jonas-thesis}. Recently, Krivelevich and
	Kronenberg~\cite{Krivelevich-Kronenberg} also studied the same problem 
	independently and used different strategies to obtain the same sharp thresholds. 
	In their paper they also deal with the $k$-connectivity game for arbitrary
	constant $k$, which we only consider here for $k=1$.
	For the perfect matching and Hamiltonicity games our 
	strategy for \CM succeeds much faster, 
	with wasting only
	$O(\log n)$ extra moves above the size of a winning set, as opposed to the
	$O(n^{\alpha})$ in \cite{Krivelevich-Kronenberg}. 

	\subsection{Terminology and organization}
	
	We will use the following terminology and conventions.
	A \textit{move} consists of claiming one edge. \textit{Turns} are taken alternately, 
	one turn can have multiple moves. For example: With a $(1:b)$ bias,
	Maker has $1$ move per turn, while Breaker has $b$ moves. 
	A {\em round} consists of a turn by Maker followed by a turn by Breaker. 
	By a strategy we mean a set of rules which specifies what the
	player does in any possible game scenario. For technical reasons 
	we {\em always} consider strategies that last until there are no free
	edges. This will be so even if the player has already won, already
	lost, or his strategy description includes ``then he forfeits''; in
	these cases the strategy just always occupies an arbitrary free edge,
	say with the smallest index. 
	The {\em play-sequence} $\Gamma$ of length $i$ of an actual game between Maker and Breaker 
	is the list $(\Gamma_1, \ldots , \Gamma_i) \in E(K_n)^i$ of the first $i$ edges that were occupied during the game by either of the players, 
	in the order they were occupied. 
	We make here the convention that a player with a bias $b> 1$ occupies 
	his $b$ edges within one turn in succession and these are noted in the play-sequence 
	in this order
	(even though in the actual game it makes no difference in what order one player's moves are
	occupied within one of his turns). 
	We denote Maker's graph after $t$ turns with $G_{M,t}$ and similarly Breaker's graph with $G_{B,t}$. Note that these graphs have $at$ and $bt$ edges respectively. 
	We will use the convention that Maker goes first. This is more of a notational 
	convenience, since the proofs can be easily adjusted to Breaker going first, and yielding 
	the same asymptotic results. 
	We will routinely omit rounding signs, whenever they are not crucial in affecting our asymptotic statements.
	
	We first introduce the notion of a \textit{permutation strategy} in the next section, and continue to prove Theorems \ref{thm:resPmCMRB} and \ref{thm:resHamCMRB} in Section \ref{sec:CMRB}.
	
	\section{The permutation strategy}\label{sec:permutation}
	In this section, we discuss an alternative way to think of half-random games that will simplify our reasoning in many proofs. This discussion does not depend on the game's win conditions, so we will refer to the random player and the clever player as \RP and \CP respectively, regardless of who is Breaker and Maker. This allows us to also use Proposition \ref{prop:permutationStrategy} in \cite{THRG1}. 
	We refer to \RP's and \CP's bias as $r$ and $b$ respectively.
	
	One reason why half-random games are more difficult to study than fully random games, is that \RP's graph is in fact not fully random. This is because the edges occupied by \CP can no longer be claimed by \RP, the deterministic and random aspects of the game interact. Our goal in this section is to relate \RP's graph to a fully random auxiliary graph, so we can apply results from the rich theory of random graphs still.
	
	Given a permutation $\sigma \in S_{E(K_n)}$ of the edges of $K_n$, i.e. $\sigma: \left[ {n\choose 2} \right] \rightarrow E(K_n)$, a player can use $\sigma$ to determine a strategy as follows. We say that he follows the \emph{permutation strategy} $\sigma$ if for every move, he scans through the edges in $\sigma$ and occupies the first one that is free. That is, he occupies the edges in their order in $\sigma$, skipping the ones occupied by his opponent. This naturally leads to a randomized strategy for \RP: in the beginning, she picks $\sigma$ uniformly at random out of all permutations, and then follows the corresponding permutation strategy. As it turns out, this is equivalent to the original method by which \RP chooses her moves (i.e., always choosing a uniformly random free edge). We formalize this in the following proposition and include a proof for completeness.	
	\begin{prop}\label{prop:permutationStrategy}
		For every strategy $S$ of \CP in a $(r:c)$-game on $E(K_n)$ the following is true.
		For every 
		$m\leq {n\choose 2}$
		and every sequence 
		$\Gamma = (\Gamma_1, \ldots , \Gamma_m)$ of distinct edges,
		the probability that 
		$\Gamma$ is the play-sequence of a half random game between 
		\CP  playing according to strategy $S$ and \RP
		is equal to the probability that $\Gamma$ is the play-sequence of the game when 
		$\RP$ plays instead according to the random permutation strategy.
	\end{prop}
\begin{proof}
		Let $R\subseteq [m]$ and $C = [m]\setminus R$ be the subsets of coordinates in any play-sequence of length $m$, 
		which belong to \RP's and \CP's moves in an $(r:c)$-biased game, respectively.
		Note that these sets are determined by $m, r,$ and $c$ and by who starts the game (and independent of the play-sequence).
		
		Let $\Gamma =  (\Gamma_1, \ldots , \Gamma_m)$ be a sequence of distinct edges which can be realized as 
		a play-sequence provided \CP plays according to strategy $S$ (otherwise the probability of $\Gamma$ is $0$ in both games). 
		In other words, for every $j\in C$, if $(\Gamma_1, \ldots , \Gamma_{j-1})$ is a play-sequence of the $(r:c)$-game then the next edge
		\CP chooses according to $S$ is $\Gamma_j$.
		
		Clearly, the probability that this particular $\Gamma$ 
		is the play-sequence of the half-random game is 
		\begin{equation} \label{eq:technical-count}
		\prod_{j\in R}\frac{1}{{n\choose 2} - j+1}.
		\end{equation}

		Let us now turn to the game generated by the random permutation strategy 
		and let us define ${\cal N}(\Gamma, S) = {\cal N}$ to be the set
		of those permutations 
		$\sigma\in S_{E(K_n)}$ which produce the play-sequence $\Gamma$ when \RP plays with the permutation strategy
		$\sigma$ against \CP's strategy $S$.
		Then ${\cal N}$ consists exactly of those permutations $\sigma$ for which
		\begin{itemize} 
			\item[(1)] the relative order of the edges in $\{ \Gamma_i  : i\in R\}$
			agrees in $\sigma$ and $\Gamma$
			\item[(2)] the edges in $\{ \Gamma_i : i\in R\}$ precede in $\sigma$ the edges in 
			$E(K_n) \setminus \{ \Gamma_i : i\in [m]\}$. 
			\item[(3)] For every $j\in C$, the Clever-edge $\Gamma_j$ comes after 
			all the Random-edges  $\{\Gamma_i : i\in R, i< j\}$ in $\sigma$.
		\end{itemize}
		We can obtain every such permutation by starting exactly with the 
		restriction of $\Gamma$ to $R$, so (1) is satisfied. 
		Then we append the edges from 
		$E(K_n) \setminus \{ \Gamma_i : i\in [m]\}$  in an arbitrary order, 
		so $(2)$ holds. Finally we insert 
		the Clever-edges $\Gamma_j$, $j\in C$, one by one, in decreasing order, 
		making sure that $(3)$ is maintained. 
		When inserting the edge $\Gamma_j$, the number of possible places 
		is exactly ${n\choose 2} - j +1$, since all the edges $\Gamma_l$ with $l> j$ are already 
		there and all the edges of index $l< j$ which are already there are contained in $R$ (and hence must precede $\Gamma_j$). 
		Hence the number of permutations in ${\cal N}$ is 
		$$\left({n\choose 2} - m\right)! \prod_{j\in C} \left({n\choose 2}- j+1\right).$$
		Hence the probability of ${\cal N}$ is equal to \eqref{eq:technical-count}
		since $C$ and $R$ partition $[m]$. 
	\end{proof}
	
	In the following, for $1\leq m\leq {n\choose 2}$ and a permutation $\sigma \in
	S_{E(K_n)}$, we let $\Gsigma{\sigma}{m}\subseteq K_n$ be the subgraph with edge set $E(\Gsigma{\sigma}{m}) = \{ \sigma(i) : 1\leq i \leq m\}$.
	Note that if the permutation $\sigma$ is chosen uniformly at random, then $\Gsigma{\sigma}{m}$ is distributed like the random
	graph $G(n,m)$.
	
	Let us now switch back to the \CM / \RB setup. Assume \RB plays a particular game according to a permutation
	$\sigma \in S_{E(K_n)}$, and let $m_i$ be the index in $\sigma$ of the last edge he takes
	in round $i$, i.e. that edge is $\sigma (m_i)$. Then \RB's graph after round $i$ is contained in $\Gsigma{\sigma}{m_i}$.  
	Note that $m_i \geq ir$, but since \RB maybe skipped some edges occupied by \CM, the actual value depends on the strategy
	of \CM and the 
	permutation $\sigma$ itself. However, since \CM occupied only
	$ic$ edges so far, we also have that $m_i \leq i(r+c)$.
	Hence  \RB's graph after the $i$th round is always contained in the random
	graph $\Gsigma{\sigma}{i(r+c)}$. This line of reasoning leads to the following proposition.
	
	\begin{prop}\label{prop:aux-randomgraph} Let $b$ and $i$ be positive integers such that 
		$i\leq \frac{{n\choose 2}}{b+1}$. Then for every 
		monotone increasing graph property ${\cal P}$ and 
		strategy $S$ of \CM for a $(1:b)$ half-random game the following holds.
		The probability that in a half-random game against strategy $S$ of \CM the 
		graph of \RB after the $i$th round has property 
		${\cal P}$ is at most $\prob{ G(n, i(b+1)) \mbox{ has property ${\cal P}$} }$.
	\end{prop}
	
	\begin{proof} 
		Consider all play sequences of length $i(1+b)$ that are possible
		with \CB playing according to $S$ so that by round $i$ his graph has
		property ${\cal P}$. 
		By the previous proposition the probability of these play sequences in
		the half-random game is equal to $\frac{|{\cal M}|}{{n\choose 2}!}$,
		where ${\cal M} = {\cal M}({\cal P}, i, S)$ is the set of  permutations
		$\sigma$ of $E(K_n)$ having the property that
		if \RB plays according to the permutation strategy $\sigma$ against strategy $S$ of \CM,
		then by the end of round $i$ \RB's graph has property ${\cal P}$. 
		
		Now recall that all edges of \RB's graph in the first $i$ rounds while
		playing according to an arbitrary permutation strategy $\sigma$ 
		are among the first $i(b+1)$ elements of $\sigma$.
		Therefore, since ${\cal P}$ is monotone increasing,
		for any permutation $\sigma \in {\cal M}$, the graph $\Gsigma{\sigma}{i(b+1)}$ has property ${\cal P}$.
		Since for a uniform random permutation $\sigma$, the graph 
		$\Gsigma{\sigma}{i(b+1)}$ is a uniform random graph 
		$G(n, i(b+1))$, the statement follows.
	\end{proof}

	\section{\CM vs \RB}\label{sec:CMRB}
	In this section we prove the non-trivial parts of Theorems~\ref{thm:resPmCMRB} and \ref{thm:resHamCMRB} involving \CM's strategy. 
	
	For our proofs we fix an $\epsilon>0$ sufficiently small and set the following values: 
	\begin{eqnarray}
	p &:=& 1-\frac{\epsilon}{2}\\
	k=k(n) &:=& 4\left\lceil\frac{\ln{n}}{\ln{(1/p)}}\right\rceil\\
	l &:=& 8\left\lceil\frac{1}{\epsilon \ln (1/p)}\right\rceil
	\end{eqnarray}
	Note that $k=k(n)$ is of the order $\ln n$ and $l$ is constant
	depending only on $\epsilon$.
	
	We will need a few properties of  \RB's graph, which are borrowed from the uniformly
	random model $G(n,m)$. 
	\begin{lem}\label{lemGnm} 
		Let $n, b, t\in \nats$ such that $(b+1)t\leq m:=p{n\choose 2}$ and let
		$S$ be a strategy of \CM in a $(1:b)$ half-random game. Then a.a.s. the graph $G_{B,t}$ of \RB after $t$ 
		rounds in a game against \CM playing according to $S$ has the following properties.
		\begin{enumerate}[(i)]
			\item\label{lemDegree} 
			$G_{B,t}$ has maximum degree at most $\left(1-\frac{\epsilon}{4}\right)n$.
			\item\label{lemNoClique} There is no set of $k$ vertices in $G_{B,t}$ inducing
			at least ${k\choose 2} -\frac{k}{2}$ edges.
			\item\label{lemNoBipartite} $G_{B,t}$ contains no complete bipartite graph of size $\frac{\epsilon}{8} n\times l$.
			\item\label{lemNoBipartiteBalanced} $G_{B,t}$ contains no complete bipartite graph of size $\frac{\epsilon}{32l} n\times \frac{\epsilon}{32l} n$.
		\end{enumerate}
	\end{lem}
	\begin{proof}
		We show the properties for $G(n,m)$ and then transfer them to $G_{B,t}$ using Proposition~\ref{prop:aux-randomgraph}. 
		To estimate, we repeatedly use that $\frac{{N-q \choose m-q}}{{N\choose m}} = \prod_{i=0}^{q-1} \frac{m-i}{N-i} \leq \left( \frac{m}{N}\right)^q=p^q$, where $N={n\choose 2}$.
		
		For part $(i)$ see e.g. \cite{erdrenyi60}, Theorem 10.
		
		For part $(ii)$ we have that the probability that there exists a $k$-element set $K$ such that $G(n,m)$ has at least ${k\choose 2} -\frac{k}{2}$ edges in $K$ is, by the union bound, at most
		\begin{eqnarray*} 
			{n\choose k} {{k\choose 2} \choose {k\choose 2} - \frac{k}{2}} \frac{\bigchoose{N - {k\choose 2} + \frac{k}{2}}{m-{k\choose 2} + \frac{k}{2}}}{{N \choose m}} &  \leq & 
			\left(\frac{en}{k}\right)^k \left(e(k-1)\right)^{k/2} p^{{k\choose 2} - \frac{k}{2}}\\ 
			& \leq  &e^{k(3/2+\ln n - \frac{1}{2}\ln k - \frac{k-2}{2} \ln (1/p))} =o(1)
		\end{eqnarray*}
		We prove parts $(iii)$ and $(iv)$ similarly, by observing that the
		probability of the event that there is a complete bipartite graph of size $r\times q$ in $G(n,m)$ is upper
		bounded by $${n\choose r}{n\choose q} \frac{{N-rq \choose
				m-rq}}{{N\choose m}}.$$
		For $(iii)$, we set $r=l$ and $q=\frac{\epsilon}{8}n$ and estimate
		by $n^r2^np^{qr} = e^{l\ln n + (\ln 2) n - \frac{\epsilon}{8}nl\ln
			(1/p)}.$ This tends to $0$ by the choice of $l$. 
		For $(iv)$ we set $r=q=\frac{\epsilon}{32l}n$ and estimate with
		$2^n2^n p^{\frac{\epsilon^2}{1024l^2} n^2} = o(1)$.
	\end{proof}
	
	Towards the end of both of his strategies, \CM occasionally sets out to make 
	a {\em double move} or  a {\em triple move}.  By this we mean that \CM  identifies two or three 
	free edges which he intends to occupy immediately in the next two or three rounds, respectively. 
	To occupy the first edge is of course no problem since it is free, 
	but in order to be able to occupy the second or third edge, it is also  
	necessary that \RB did not occupy them in his turn(s) in between. 
	The next simple lemma states that this is very likely if there are still many free edges.
	\begin{lem} \label{lem:double-triple} The probability that 
		\CM is not able to complete a double move (or a triple move) within the first $t$ rounds of 
		a $(1:b)$ half-random game with $b\leq n$ is at most
		$\frac{4}{\epsilon n}$ (or $\frac{12}{\epsilon n}$), provided 
		the number of free edges is at least ${n\choose 2} - (b+1)t \geq \frac{\epsilon}{4}n^2$.
	\end{lem}
	\begin{proof}The probability, that out of the still at least
		$\frac{\epsilon}{4}n^2$ free edges \RB occupies exactly the second edge
		of the double move \CM has just started, is at most
		$\frac{4}{\epsilon n^2}$. He has at most $n$ chances before $\CM$ completes 
		his double move, hence the upper bound follows. For triple moves $\RB$
		has $n$ chances to occupy the second edge of the triple move and $2n$
		chances for the third edge. 
	\end{proof}
	
	\subsection{\CM builds a perfect matching}
	In this section we prove the non-trivial part of Theorem~\ref{thm:resPmCMRB}.
	We consider the $(1:b)$ perfect matching game between \CM and \RB on $E(K_n)$, 
	where $n$ is even and $b\leq (1-\epsilon)n$ for an arbitrary but fixed $\epsilon, 0< \epsilon < 1/2$.
	Throughout the proof when we say that a 
	vertex is \textit{isolated}, we always mean that it is isolated in \CM's graph at the current point 
	in the game. 
	
	During the game \CM maintains a matching $M$ of his graph. He starts with 
	$M=\emptyset$ and then eventually achieves that $M$ is perfect, 
	at which point \CM wins the game.
	Let us call an edge of $K_n$ {\em vacant} if it is neither occupied by \RB nor used 
	by \CM in his matching $M$. For an isolated vertex $a$, we define
	\[X_a:=\left\{u\in V: au\mbox{ is vacant}\right\}\]
	to be the set of vertices with a vacant edge to $a$.  
	Further, let \[X_a^+:=\left\{v\in V: vu\in M, u\in X_a\right\}.\]
	\\
	We now define strategy $S_{PM}$ for \CM. 
	If $S_{PM}$ calls \CM to take an edge he has already
	occupied, he takes an arbitrary free edge. If anytime during a game 
	\CM is not able to make a move according to the directions below, we
	say that he \emph{forfeits}. 
	(Recall that for technical reasons, \CM continues 
	to play in this case by always claiming the free edge with the
	smallest index until the board is full.)
	The strategy consists of three stages.
	\begin{enumerate}
		\item[{\bf Stage 1.}] This stage lasts while $|M| < \frac{n-k(n)}{2}$. \CM iteratively occupies 
		an arbitrary free edge $e$ between two isolated vertices, 
		and adds $e$ to $M$.
		\item[{\bf Stage 2.}]  This stage lasts until $\frac{n-k(n)}{2}\leq |M| < \frac{n-l}{2}$
		and consists of $\frac{k(n)-l}{2}$ double moves, each increasing the size of $M$ by one
		(using augmenting paths of length $3$).
		For each of his double moves \CM identifies an arbitrary 
		edge $uv\in M$ such that there exists isolated vertices 
		$a\in X_u$ and $b\in X_v$  and then he occupies 
		$au$ and $bv$ in his next two turns. 
		Finally \CM removes $uv$ from $M$, and adds $au$ and $bv$ instead.  
		\item[{\bf Stage 3}] This stage lasts until $\frac{n-l}{2} \leq |M| < \frac{n}{2}$ 
		and consists of $\frac{l}{2}$ triple moves, each increasing the size of $M$ by one
		(using augmenting paths of length $5$). 
		For each of his triple moves \CM first identifies two arbitrary isolated vertices $a,b$ and 
		then an arbitrary vacant edge $wz$ with $w\in X_a^+, z\in X_b^+$. 
		Let $u\in X_a$ and $v\in X_b$ be the vertices with $uw\in M$ and $zv\in M$, respectively. 
		In his next three turns, \CM occupies  $au$, $wz$ and $vb$. 
		He then adds these three edges to $M$, while removing $uw$ and $zv$. 
	\end{enumerate}
	
	Throughout this process $M$ remains a matching, and with each single-/double- or triple move increases in size by one.
	Thus, after all 3 stages are complete, $M$ is a matching of size $\frac{n}{2}$, i.e. a perfect matching.
	
	Therefore, what remains to show is that \CM can a.a.s. execute strategy $S_{PM}$ without forfeiting.

	\begin{proof}[Proof of Theorem~\ref{thm:resPmCMRB}.] 
		Let $\epsilon>0$ be fixed and let $b\leq (1-\epsilon)n$ be a positive integer.
		We prove that \CM, playing  against \RB in the $(1:b)$ half-random game, 
		can execute the strategy $S_{PM}$ without forfeiting, a.a.s. 
		This in particular, will imply that \CM wins 
		the $(1:b)$ half-random perfect matching game (and thus also the degree-1 game)
		within $\frac{n}{2}+O(\ln n)$ moves, a.a.s.
		
		First note that strategy $S_{PM}$ takes at most 
		\[t:=\frac{n-k(n)}{2}+k(n)-l+\frac{3l}{2}=\frac{n+k(n)+l}{2} = \frac{n}{2} + O(\ln n)\]
		rounds. This is because in Stage 1, $\frac{n-k(n)}{2}$ edges are added to $M$, and in Stage~2 $\frac{k(n)-l}{2}$ edges, taking two rounds each. This leaves $\frac{l}{2}$ edges to be added in Stage~3, which takes $\frac{3l}{2}$ rounds.
		
		Observe also that for the total number of edges claimed by either player we have 
		\[(b+1)t=\left((1-\epsilon)n+1\right)\frac{n+o(n)}{2}\leq p{n\choose 2 } = m.\]
		This has two important consequences. On the one hand the conditions of 
		Lemma~\ref{lemGnm} are satisfied, so a.a.s. all properties $(i)-(iv)$ hold for 
		the graph \RB occupies by turn $t$, and since the properties are decreasing, also at all previous points in the game. On the other hand  ${n\choose 2} - (b+1)t \geq \frac{\epsilon}{4}n^2$, 
		so by Lemma~\ref{lem:double-triple} the probability that any double or triple move 
		\CM has started cannot be completed is at most $\frac{12}{\epsilon n}$.
		Since the number of double moves is $O(\ln n)$ and the number of triple moves is $O(1)$, this will
		occur only with probability $O(\frac{\ln n}{n})$. In other words, a.a.s. \CM can complete every double or triple move he starts.
		
		We now assume that indeed these two events hold, i.e. \RB's graph has
		properties $(i)-(iv)$ of Lemma~\ref{lemGnm} up to at least round $t$,
		and \CM can complete every double or triple move he starts. We go
		through the three stages and show that under these conditions, the
		strategy can be carried through without forfeiting.
		
		First let $|M| < \frac{n-k(n)}{2}$, so we are in Stage 1. 
		Since in Stage 1 there are $\frac{n-k(n)}{2}$ rounds, there must be at least $k(n)$ isolated 
		vertices left in \CM's graph. By Property $(ii)$ of Lemma~\ref{lemGnm}, \RB has no clique of size $k(n)$ occupied, and thus there must be at least one vacant edge between two isolated vertices.

		Let now $\frac{n-k(n)}{2}\leq |M| < \frac{n-l}{2}$, so we are in Stage 2. 
		Let $a$ be an arbitrary isolated vertex and consider the edges
		of $K_n$ between $X_a^+$ and  the set $L$ of isolated vertices
		different from $a$. If any of these edges is vacant, \CM can start his
		double move. Otherwise \RB's graph contains a
		complete bipartite graph  of size $|L|\times |X_a^+|$. Since Stage 2 is
		not yet over, we have $|L|\geq l$. For the other side,
		$\left|X_a^+\right|= 2|M| - \deg_B(a) \geq n-k(n) -\deg_B(a) \geq
		\frac{\epsilon}{8}n$ by Property $(i)$ of Lemma~\ref{lemGnm} and since $k(n) = O(\ln n)$. 
		Since by Property $(iii)$ of Lemma~\ref{lemGnm} \RB's graph contains no
		complete bipartite graph $K_{\frac{\epsilon}{8}n, l}$, one of the
		edges between $L$ and $X_a^+$ must be vacant, allowing \CM to start his double move. 
		
		Let now $\frac{n-l}{2}\leq |M| < \frac{n}{2}$, so we are in Stage 3. Here \CM has to select
		triplets of moves. For this there needs to be two isolated vertices $a, b$ such that 
		there is a vacant edge $wz$ with $w\in X_a^+, z\in X_b^+$. 
		Since $M$ is not yet perfect and $n$ is even, there must be two isolated vertices $a$ and 
		$b$. As in Stage 2, we also have that both the sizes $\left|X_a^+\right|$ and $\left|X_b^+\right|$  
		are at least  $2|M| - \Delta (G_B) \geq n- l - \Delta (G_B)\geq \frac{\epsilon}{8} n$ (recall that $l = O(1)$). 
		In particular, there are disjoint sets $Y_a^+\subset X_a^+$ and $Y_b^+\subset X_b^+$ of size at least $\frac{\epsilon}{16} n$ each. By Property $(iv)$ of Lemma~\ref{lemGnm}, \RB's graph contains no complete bipartite graph $K_{\epsilon n/16,\epsilon n/16}$, which means that indeed there is a vacant edge $wz$ with $w\in X_a^+, z\in X_b^+$.
	\end{proof}

	\subsection{\CM builds a Hamilton cycle}
	\label{sec:CMRBHam}
	
	We now turn towards the Hamiltonicity game and show the non-trivial direction of Theorem~\ref{thm:resHamCMRB}. 
	Recall the values $p$, $k(n)$ and $l$ as defined above.
	
	First let us describe \CM's strategy informally. The analysis uses many ideas from the perfect matching game. 
	Actually,  first \CM follows the strategy $S_{PM}$ to build a perfect matching $M$ in 
	$\frac{n}{2}+O(\ln n)$ moves.
	Next, \CM performs another sequence of similar steps, using the matching $M$ as a starting point, 
	and connecting its edges first to a Hamilton path and then a cycle.  
	The central structure \CM maintains will be a sequence ${\cal P}_i$, $i=\frac{n}{2}, \ldots, 1$,
	of families of paths of Maker's graph, such that the paths of each family ${\cal P}_i$ 
	partition the vertex set.
	To start \CM sets ${\cal P}_{\frac{n}{2}}:=M$ to be a set of $\frac{n}{2}$ paths of length $1$.  
	Then \CM performs a sequence of 
	$\frac{n}{2}$ single, double, or triple moves. Each of these moves reduces the number of
	paths in the family by one, hence ${\cal P}_{1}$ contains a single Hamilton path. 
	In his last triple move \CB closes this path to a Hamilton cycle.
	Similarly to the perfect matching game, the number of double moves of \CB will be $O(\ln n)$ 
	and the number of triple moves will only be $O(1)$. Hence the game lasts at most $n + O(\ln n)$ rounds.  
	For convenience in notation we will assume that $n$ is even, the odd
	case can be handled similarly: \CM first occupies a matching of
	size $\frac{n-1}{2}$, then connects the lone isolated vertex to an
	arbitrary matching edge and thus builds his
	initial family of nontrivial paths ${\cal P}_{\frac{n-1}{2}}$ covering the vertex set.
	
	In the following, for a path $\gamma\in {\cal P}_i$, we write $\gamma$ as a sequence of vertices 
	$\gamma_0,\dots,\gamma_{s(\gamma)}$, where $s(\gamma)$ denotes the length of 
	$\gamma$. We use a fixed direction on the ordering, for example, we can demand that 
	$\gamma_0<\gamma_{s(\gamma)}$, when seeing the vertices as elements in $[n]$. 
	For a vertex $a$, we define the following helpful set,
	consisting of those vertices which are followed by a vertex of $X_a$ on a path of ${\cal P}_i$
	(recall that $X_a$ denotes the set of vertices with a vacant edge to $a$).
	\begin{align*}
	X_a^\leftarrow&:=\{\gamma_{j-1} : \gamma \in {\cal P}_i, \gamma_j \in X_a\}\setminus\{a\}
	\end{align*}
	
	We now describe \CM's strategy $S_{HAM}$.
	\begin{description}
		\item[Stage 0] Build a perfect matching $M$ in $\frac{n}{2} +O(\ln n)$ moves using strategy $S_{PM}$. Set ${\cal P}_{\frac{n}{2}} = M$.
		\item[Stage 1] Let $\frac{n}{2}\geq i > k(n)$. To construct ${\cal P}_{i-1}$ 
		from ${\cal P}_i$ \CM uses a single move to occupy a vacant edge $e$ between two
		endpoints $a$ and $b$ that belong to two different paths 
		$\alpha$ and $\beta \in {\cal P}_i$.  
		He obtains ${\cal P}_{i-1}$ by removing $\alpha$ and $\beta$ from ${\cal P}_i$, 
		and adding the new path obtained by connecting $\alpha$ and $\beta$ with $e$. 
		(Again, with vacant we mean neither occupied by \RB nor used by \CM on the paths of ${\cal P}_i$; 
		if \CM has the edge previously occupied but is not using it, he just starts using it and occupies an arbitrary edge somewhere else.)
		\item[Stage 2] Let $k(n)\geq i > l$. 
		To construct ${\cal P}_{i-1}$ from ${\cal P}_i$ \CM uses a double move.
		Let us a fix the starting vertex $a:=\alpha_0$ of an arbitrary path $\alpha\in {\cal P}_i$. 
		Let 
		$B:= \{\beta_{s(\beta)}:\beta\in {\cal P}_i\setminus\{\alpha\}\}$ be the set of endpoints of 
		the paths in ${\cal P}_i$ other than $\alpha$.
		\CB  then identifies a vertex $v\in X_a^\leftarrow$ and a vertex $b\in B$ such that the edge $bv$ is currently vacant. Let $u \in X_a$ be the neighbor that follows $v$ on the path 
		$\gamma\in {\cal P}_i$ which contains $v$, say $u=\gamma_j$, 
		$v=\gamma_{j-1}$.  
		\CM now occupies the edges $au$ and $bv$ in his next two turns.
		The new family ${\cal P}_{i-1}$ depends on which of the following three cases hold.
		(Recall that $\alpha\neq \beta$). 
		\begin{description}
			\item[Case 1: $\gamma\neq \alpha,\beta$.] Then \CM obtains ${\cal P}_{i-1}$ by removing $\alpha$, $\beta$ and $\gamma$ from ${\cal P}_i$, and adding 
			\[\alpha_{s(\alpha)}\dots\alpha_0\gamma_j\gamma_{j+1}\dots\gamma_{s(\gamma)} 
			\mbox{ and } \gamma_0\gamma_1\dots\gamma_{j-1}\beta_{s(\beta)}\dots\beta_0.\]
			\item[Case 2: $\gamma= \alpha$.] I.e. $u=\alpha_j$, $v=\alpha_{j-1}$. Then \CM obtains ${\cal P}_{i-1}$ by removing $\alpha$ and $\beta$ from ${\cal P}_i$, and adding \[\alpha_{s(\alpha)}\dots\alpha_{j+1}\alpha_j\alpha_0\alpha_1\dots\alpha_{j-1}\beta_{s(\beta)}\dots\beta_0.\]
			\item[Case 3: $\gamma= \beta$.] I.e. $u=\beta_j$, $v=\beta_{j-1}$. Then \CM obtains ${\cal P}_{i-1}$ by removing $\alpha$ and $\beta$ from ${\cal P}_i$, and adding \[\alpha_{s(\alpha)}\dots\alpha_0\beta_j\beta_{j+1}\dots\beta_{s(\beta)}\beta_{j-1}\beta_{j-2}\dots\beta_0.\]
		\end{description}
		\item[Stage 3] Let $l\geq i > 1$. To construct ${\cal P}_{i-1}$ from ${\cal P}_i$ \CM uses a triple move. He first identifies two arbitrary paths 
		$\alpha,\beta\in {\cal P}_i$, $\alpha\neq\beta$. Let $a:=\alpha_0$, $b:=\beta_0$.
		Next, he sets $\gamma^{\tt{a}}\in {\cal P}_i$ to be a path such that 
		$|X_a\cap\gamma^{\tt{a}}|$ is maximal, and defines $\gamma^{\tt{b}}$ 
		similarly. Then he constructs vertex sets $X_a^*$ and $X_b^*$ depending on two cases:
		\begin{description}
			\item[Case 1: Neither $\gamma^{\tt{a}} =\gamma^{\tt{b}}=\alpha$ nor 
			$\gamma^{\tt{a}} =\gamma^{\tt{b}}=\beta$.] Then we simply define 
			$X_a^*:= X_a^\leftarrow\cap\gamma^{\tt{a}}$ and $X_b^*:= X_b^\leftarrow\cap\gamma^{\tt{b}}$.
			\item[Case 2: $\gamma^{\tt{a}}=\gamma^{\tt{b}}=\alpha$ 
			or $\gamma^{\tt{a}} =\gamma^{\tt{b}}=\beta$.]
			Let us assume $\gamma^{\tt{a}}=\gamma^{\tt{b}}=\alpha$, 
			the other case is treated similarly. First, we write 
			\begin{align*}
			X_a\cap \alpha &= \{\alpha_{e_0},\dots,\alpha_{e_q}\} \\
			X_b\cap \alpha &= \{\alpha_{f_0},\dots,\alpha_{f_r}\}
			\end{align*}
			Then, we define
			\begin{align*}
			X_a^*&:=\begin{cases} \{\alpha_{e_1-1},\dots,\alpha_{e_{\frac{q}{2}} -1}\} & \text{if $e_{\frac{q}{2}}<f_{\frac{r}{2}}$}\\
			\{\alpha_{e_{\frac{q}{2}}-1},\dots,\alpha_{e_q -1}\}& \text{if $e_{\frac{q}{2}}\geq f_{\frac{r}{2}}$}
			\end{cases}\\
			X_b^*&:=\begin{cases} \{\alpha_{f_{\frac{r}{2}}+1},\dots,\alpha_{f_{r-1} +1}\} & \text{if $e_{\frac{q}{2}}<f_{\frac{r}{2}}$}\\
			\{\alpha_{f_{1}-1},\dots,\alpha_{f_{\frac{r}{2}}-1}\}& \text{if $e_{\frac{q}{2}}\geq f_{\frac{r}{2}}$}
			\end{cases}\\
			\end{align*}
		\end{description}
		Now let $w\in X_a^*$ and $z\in X_b^*$ be such that $wz$ is a vacant edge. Then let $u\in X_a$ be 
		the neighbor following $w$ on $\gamma^{\tt{a}}$ and $v\in X_b$ be the neighbor
		following $z$ on $\gamma^{\tt{b}}$. In his next three moves \CM claims the edges $au$, 
		$bv$ and $wz$. He updates his paths by adding these three edges, and removing the edges 
		$uw$ and $vz$. 
		It is now easy to verify that this indeed reduces the number of paths in ${\cal P}_i$ by one in 
		each case (see Figure \ref{fig:casesStage3}).
		
		\begin{figure}[h!]
			\centering
			\includegraphics[width=0.3\textwidth]{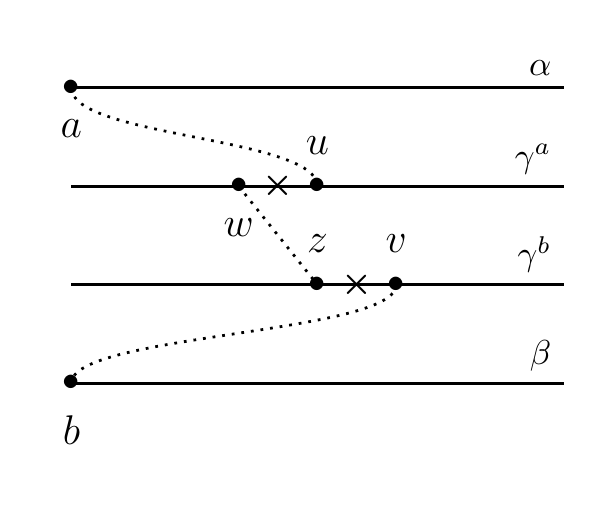}
			\includegraphics[width=0.3\textwidth]{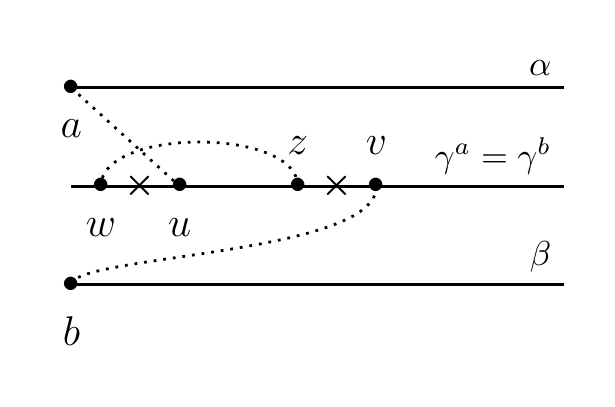}
			\includegraphics[width=0.3\textwidth]{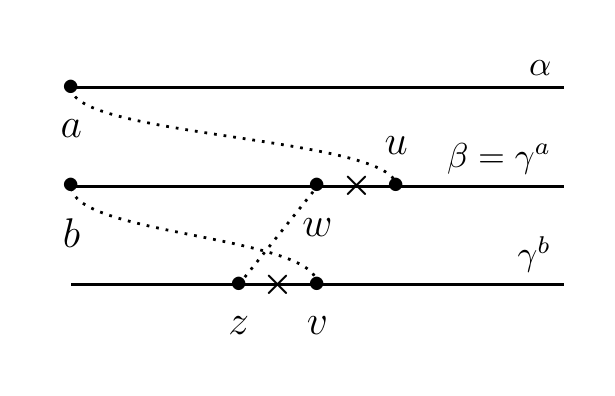}\\
			\includegraphics[width=0.3\textwidth]{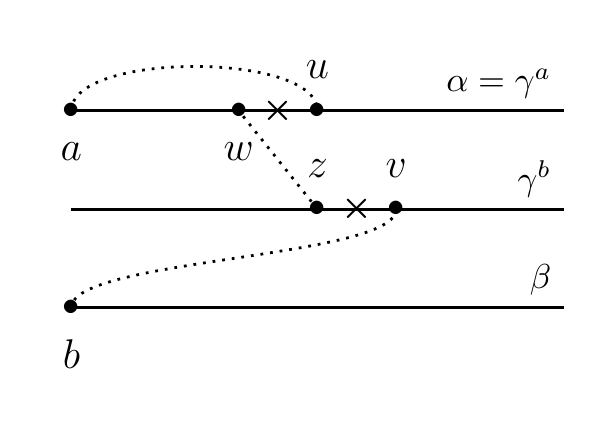}
			\includegraphics[width=0.3\textwidth]{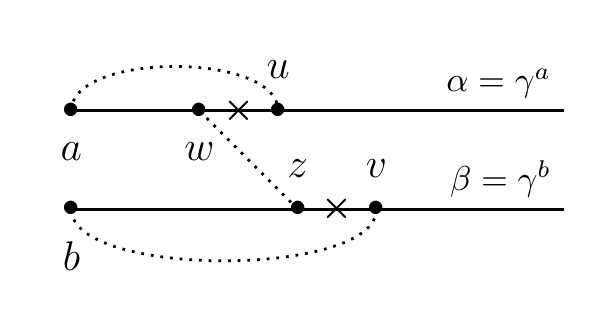}
			\includegraphics[width=0.3\textwidth]{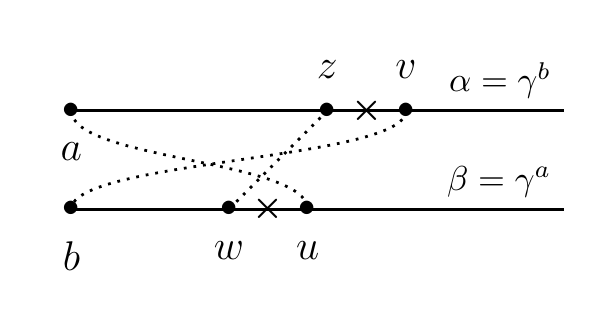}\\
			\includegraphics[width=0.45\textwidth]{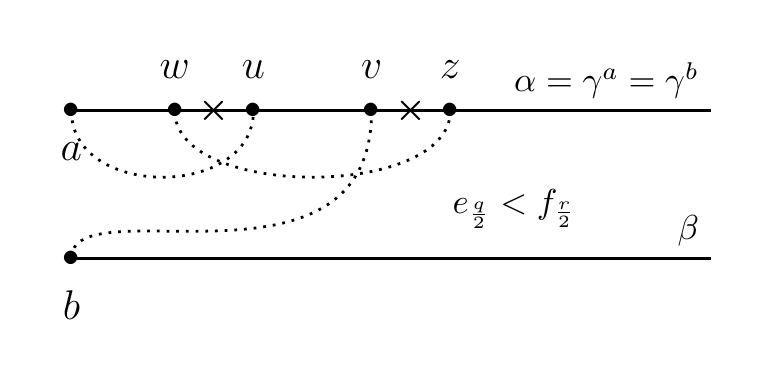}  
			\includegraphics[width=0.45\textwidth]{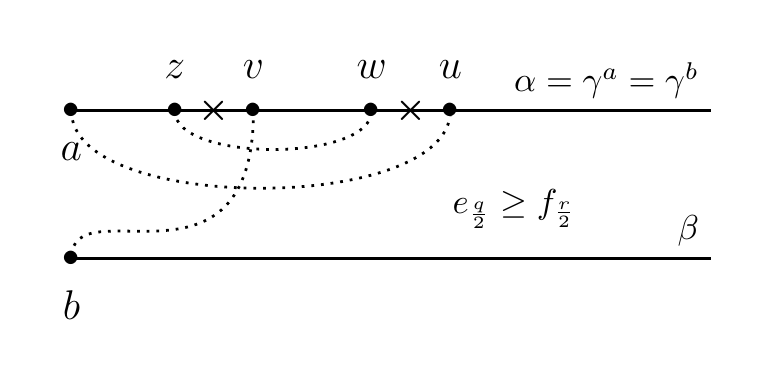}     
			\caption{The different sub-cases of Stage 3, with the two sub-cases of Case 2 on the bottom. Symmetric cases where the roles of $\alpha$ and $\beta$ are swapped are not pictured.}
			\label{fig:casesStage3}
		\end{figure}
		
		\item[Stage 4] Assume now there is only one path $\gamma_0\dots \gamma_n$ left that covers all vertices, and has endpoints $a=\gamma_0$ and $b=\gamma_n$. We can then write $X_a=\{\gamma_{i_0},\gamma_{i_1},\dots,\gamma_{i_s}\}$ and $X_b=\{\gamma_{j_0},\gamma_{j_1},\dots,\gamma_{j_t}\}$ with $i_x<i_{x+1}$ and $j_y<j_{y+1}$ for all $x,y$. Then, if $i_{\frac{s}{2}}<j_{\frac{t}{2}}$, we take $u=\gamma_i\in X_a$ and $v=\gamma_j\in X_b$ such that $i\leq\frac{s}{2}$, $j\geq \frac{t}{2}$, and the edge $\gamma_{i-1}\gamma_{j+1}$ is vacant. Then \CM claims the edges $au$, $\gamma_{i-1}\gamma_{j+1}$ and $bv$ to complete a Hamilton cycle (deleting the edges $u\gamma_{i-1}$ and $v\gamma_{j+1}$), see Figure \ref{fig:casesStage4}. If $i_{\frac{s}{2}}\geq j_{\frac{t}{2}}$, \CM instead chooses $i\geq i_{\frac{s}{2}}$, $j<j_{\frac{t}{2}}$ such that the edge $\gamma_{i+1}\gamma_{j-1}$ is vacant and proceeds accordingly.
		\begin{figure}[h!]
			\centering
			\begin{subfigure}[b]{0.45\textwidth}
				\trimbox{0cm 0.8cm 0cm 0cm}{ 
					\includegraphics[width=\textwidth]{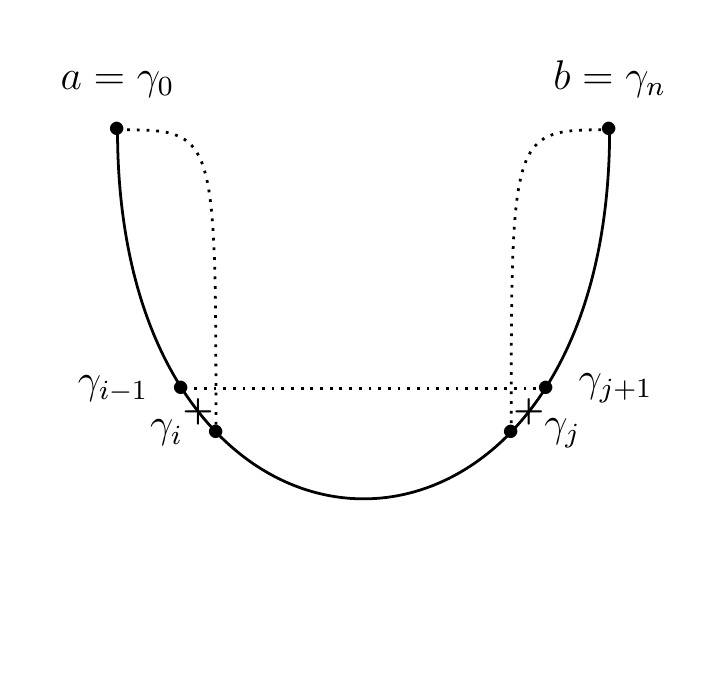}
				}
				\caption{Case $i_{\frac{s}{2}}< j_{\frac{t}{2}}$}
			\end{subfigure}
			\begin{subfigure}[b]{0.45\textwidth}
				\trimbox{0cm 0.8cm 0cm 0cm}{ 
					\includegraphics[width=\textwidth]{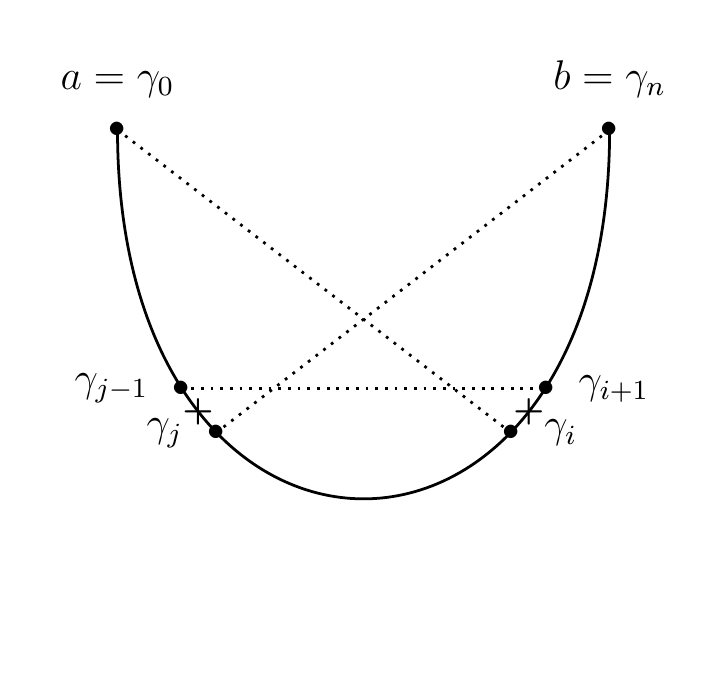}
				}
				\caption{Case $i_{\frac{s}{2}}\geq j_{\frac{t}{2}}$}
			\end{subfigure}
			\caption{Closing the Hamilton Cycle in Stage 4}
			\label{fig:casesStage4}
		\end{figure}
	\end{description}
	
	If \CM can not make a move according to these directions, he forfeits.
	
	\FloatBarrier
	
	\begin{thm}
		Let $\epsilon>0$ be fixed and let $b\leq (1-\epsilon)\frac{n}{2}$ be a positive integer.
		Then \CM, playing  against \RB in the $(1:b)$ half-random game, 
		can execute the strategy $S_{HAM}$ without forfeiting, a.a.s. Hence \CM wins 
		the $(1:b)$ half-random Hamiltonicity game  (and thus also the degree-2 game)
		within $\frac{n}{2}+O(\ln n)$ moves, a.a.s.
	\end{thm}
	
	\begin{proof}[Proof of Theorem~\ref{thm:resHamCMRB}.] 
		Let $\epsilon>0$ be fixed and let $b\leq (1-\epsilon)\frac{n}{2}$ be a positive integer.
		We will prove that \CM, playing  against \RB in the $(1:b)$ half-random game, 
		can execute the strategy $S_{HAM}$ without forfeiting, a.a.s. This, in
		particular, will imply that \CM wins 
		the $(1:b)$ half-random Hamiltonicity game  (and thus also the
		degree-2 game and the connectivity game)
		within $\frac{n}{2}+O(\ln n)$ moves, a.a.s.
		
		First note that in Stage 0 
		\CM can a.a.s. create a perfect matching within $\frac{n}{2}+O(\ln n)$ rounds by 
		Theorem~\ref{thm:resPmCMRB}.  Stage $1.$-$4.$ take another 
		$$\frac{n}{2} -k(n) + 2(k(n) -l) + 3l + 3 =  \frac{n}{2}+O(\ln n)$$ rounds, for a 
		total of $t=n+ O(\ln n)$ rounds. 
		This means that throughout the game there are at most $(b+1)t\leq\left((1-\epsilon)\frac{n}{2}+1\right)\left(n+o\left(n\right)\right)\leq p{n\choose 2}$ occupied edges, and there are always at least $\frac{\epsilon}{4}n^2$ free edges, so both 
		Lemma~\ref{lemGnm} and  Lemma~\ref{lem:double-triple} are applicable. 
		
		As in the proof of Theorem~\ref{thm:resPmCMRB}, the number of double and triple moves is $O(k(n))= O(\ln n)$,
		so the overall probability of \CM forfeiting because he could not complete a double or triple move is
		$O(\ln n/n)$. Again we assume that \RB's graph has Properties $(i)$-$(iv)$ of Lemma~\ref{lemGnm}, and \CM can complete all his double and triple moves.
		Now we need to check that the vacant edges required by $S_{HAM}$ 
		for the single, double, and triple moves of \CM do exist each time. 
		
		\begin{description}
			\item[Stage 1:] If there was no vacant edge between two endpoints of different paths, then 
			\RB would have occupied a clique of size $k(n)$ minus a matching in his graph, spanned by the endpoints 
			of the paths in ${\cal P}_i$ (where the matching consists of the edges between the two 
			endpoints of each paths). However, by Property $(ii)$, this is impossible.
			\item[Stage 2:] By Property $(i)$, 
			$\left|X_a^\leftarrow\right|\geq \left|X_a\right| - |{\cal P}_i| -1 \geq n - 1 - deg_B(a) - k(n) -1 \geq 
			\frac{\epsilon}{8}n$.
			Furthermore, $B$ has one vertex from each path in  ${\cal P}_i\setminus\{\alpha\}$.
			This means $|B| \geq l$ since the number of paths in Stage 2 is at least $l+1$.
			By Property $(iii)$, there is no $\frac{\epsilon }{8}n\times l$ 
			complete bipartite graph in $G_{B,t}$, and
			hence \CM can start his double move.
			\item[Stage 3:] For \CM being able to identify its triple move there must only be a vacant edge between
			$X_a^*$ and $X_b^*$. Both sets have linear
			size: Indeed, both $X_a$ and $X_b$ have size at least $\frac{\epsilon n}{4}-1$ by Property $(i)$ of Lemma~\ref{lemGnm}, and since 
			there are only at most $l$ paths left in this stage, $X_a\cap\gamma^{\tt{a}}$ 
			and $X_b\cap\gamma^{\tt{b}}$ both must have at least $\left(\frac{\epsilon n}{4}-1\right)/l$ vertices. Furthermore, in all cases 
			$\left|X_a^*\right|\geq\frac{1}{2}\left|X_a\cap\gamma^{\tt{a}}\right|-1$ and 
			$\left|X_b^*\right|\geq\frac{1}{2}\left|X_b\cap\gamma^{\tt{b}}\right|-1$, 
			which means that $X_a^*$ and $X_b^*$ are of size at least
			$\frac{\epsilon n}{16l}$. In particular, there are disjoint sets
			$Y_a^*\subset X_a^*$ and $Y_b^*\subset X_b^*$ of size at least
			$\frac{\epsilon n}{32l}$ each. Then by Property $(iv)$, \RB could not have occupied all edges between $Y_a^*$ and $Y_b^*$, i.e. one must be vacant. Note that by definition of $X_a^*$ and $X_b^*$, no edge between the sets is used in a path in ${\cal P}_i$ (this corresponds to the edge $wz$ in Figure \ref{fig:casesStage3}).
			\item[Stage 4:] The analysis here is very similar to stage $3$.
		\end{description}

	\end{proof}
	
	\section{Conclusion and Open Problems}
	
	We found that in the \CM-\RB scenario the trivial upper bound on the threshold bias, provided by the size of a winning set, gives the true asymptotics. 
	It would be interesting to decide whether a stronger lower bound holds. For a $k$-uniform graph property 
	${\cal F} \subseteq 2^{E(K_n)}$ let $b_{triv} = \lfloor {n\choose 2}/k \rfloor -1$ be the largest bias $b$ such that 
	Maker occupies at least $k$ edges in the $(1:b)$ game. Is it true that already 
	for \RB -bias  $b_{triv} - \omega(1)$, where $\omega(1)$ is a function  
	tending to infinity arbitrarily slowly, $\CM$ has a strategy that is winning against 
	\RM a.a.s.? 
	
	A possible first step in this direction could be to give a strategy for $\CM$ for every 
	$\epsilon > 0$ that a.a.s occupies a winning set $F\in {\cal F}$ in exactly
	$|F|$ moves against a \RB bias of $(1- \epsilon) b_{triv}$.
	We are not that far away from this: our strategies for \CM in the perfect matching and Hamiltonicity game	use only $O(\log n)$ more moves than necessary.

	\bibliography{mainLiterature}{}
	\bibliographystyle{plain}
	
\end{document}